\documentclass[11pt,a4paper]{article}

\usepackage{a4wide}
\usepackage{amsmath}
\usepackage{amssymb}
\usepackage{booktabs}       
\usepackage{algorithm}
\usepackage{algorithmic}
\usepackage{bm}
\usepackage{lscape}

\usepackage[utf8]{inputenc} 
\usepackage[T1]{fontenc}    
\usepackage{hyperref}       

\usepackage{amsfonts}       
\usepackage{graphicx}       
\usepackage{cleveref}
\usepackage{amsthm}
\usepackage{url}            
\usepackage{cite}
\usepackage{mathrsfs}

\usepackage{nicefrac}       
\usepackage{microtype}      
\usepackage{fancyhdr}       

\allowdisplaybreaks

  \makeatletter
    
    \@addtoreset{equation}{section}
  \makeatother

  \makeatletter
    
    \@addtoreset{algorithm}{section}
  \makeatother

  \makeatletter
    
    \@addtoreset{figure}{section}
  \makeatother

  \makeatletter
    
    \@addtoreset{table}{section}
  \makeatother

\def\St{\mathop{\rm St}\nolimits}
\def\tr{\mathop{\rm tr}\nolimits}
\def\diag{\mathop{\rm diag}\nolimits}

\def\skew{\mathop{\rm skew}\nolimits}
\def\sym{\mathop{\rm sym}\nolimits}
\def\grad{\mathop{\rm grad}\nolimits}
\def\Skew{\mathop{\rm Skew}\nolimits}

\def\c{{\rm c}}
\def\D{{\rm D}}

\def\Hess{\mathop{\rm Hess}\nolimits}

\def\qf{\mathop{\rm qf}\nolimits}

\def\vec{\mathop{\rm vec}\nolimits}
\def\veck{\mathop{\rm veck}\nolimits}

\def\E{{\rm E}}
\def\opt{{\rm opt}}
\def\off{{\rm off}}

\def\J{{\rm J}}
\def\N{{\rm N}}
\def\app{{\rm app}}
\def\rand{{\rm rand}}

\providecommand{\norm}[1]{\lVert#1\rVert}

\newtheorem{Prop}{Proposition}[section]

\newtheorem{Proof}{Proof}[section]
\newtheorem{Prob}{Problem}[section]

\newtheorem{predfn}{Definition}[section]
\newtheorem{prerem}[predfn]{Remark}

\title{Riemannian Newton-type methods for joint diagonalization\\on the Stiefel manifold with application to\\independent component analysis\footnotetext{\textbf{Funding:} This work was funded by JSPS KAKENHI Grant number JP16K17647.}}

\author{Hiroyuki Sato\thanks{Department of Applied Mathematics and Physics, Kyoto University, Kyoto, Japan\newline ({\tt hsato@i.kyoto-u.ac.jp}).}}

\date{}

\begin{document}
\maketitle

\begin{abstract}
The joint approximate diagonalization of non-commuting symmetric matrices is an important process in independent component analysis.
This problem can be formulated as an optimization problem on the Stiefel manifold that can be solved using 
Riemannian optimization techniques.
Among the available optimization techniques, this study utilizes the Riemannian Newton's method for the joint diagonalization problem on the Stiefel manifold, which has quadratic convergence.
In particular, the resultant Newton's equation can be effectively solved by means of the Kronecker product and the vec and veck operators, which reduce the dimension of the equation to that of the Stiefel manifold.
Numerical experiments are performed to show that the proposed method improves the accuracy of the approximate solution to this problem.
The proposed method is also applied to independent component analysis for the image separation problem.
The proposed Newton method further leads to a novel and fast Riemannian trust-region Newton method for the joint diagonalization problem.
\end{abstract}


\section{Introduction}
The joint diagonalization (JD) problem for $N$ real $n\times n$ symmetric matrices $A_1,A_2,\ldots,A_N$ is often considered on the orthogonal group $O(n)$.
The problem is to find an $n\times n$ orthogonal matrix $X$ that minimizes the sum of the squared off-diagonal elements, or equivalently, maximizes the sum of the squared diagonal elements of $X^TA_lX,\ l=1,2,\ldots,N$ \cite{cardoso1993blind}.
For more information regarding finding non-orthogonal matrices, see \cite{cardoso1993blind}.
The solution to the JD problem is valuable for independent component analysis (ICA) and the blind source separation problem  \cite{afsari2004some,cardoso1999high,cardoso1993blind,cichocki2002adaptive,Theis:2009:SDR:1532023.1532070,comon1994independent,moreau2001generalization} because solving the JD problem leads to the diagonalization of cumulant matrices of signals.

Until now, several approaches have been proposed in the context of Jacobi methods \cite{bunse1993numerical,cardoso1999high,cardoso1993blind} and Riemannian optimization \cite{Theis:2009:SDR:1532023.1532070,yamada2003orthogonal}.
In \cite{Theis:2009:SDR:1532023.1532070}, the JD problem is considered on the Stiefel manifold $\St(p,n):=\left\{Y\in\mathbb R^{n\times p}\,|\,Y^TY=I_p\right\}$ with $p\le n$.
That is, the required matrix is a rectangular matrix whose columns are orthonormal.
The orthogonal group $O(n)$ is a special case of the Stiefel manifold because $O(n)=\St(n,n)$.

Riemannian optimization refers to optimization on Riemannian manifolds.
Unconstrained optimization methods in Euclidean space such as steepest descent, conjugate gradient, and Newton's methods have been generalized to those on Riemannian manifolds \cite{AbsMahSep2008,edelman1998geometry,helmke1994optimization,smith1994optimization}.
In applying such methods,
Manopt, a MATLAB toolbox for optimization on manifolds, is available~\cite{manopt}.
In \cite{Theis:2009:SDR:1532023.1532070}, the Riemannian trust-region method is applied to the JD problem on the Stiefel manifold $\St(p,n)$.
With $Y$ varying on $\St(p,n)$ with $p<n$, minimizing the sum of the squared off-diagonal elements of $Y^TA_lY,\ l=1,2,\ldots,N$ is no longer equivalent to maximizing the sum of the squared diagonal elements.
According to \cite{Theis:2009:SDR:1532023.1532070}, the JD problem on the Stiefel manifold maximizes the sum of the squared diagonal elements of $Y^TA_lY,\ l=1,2,\ldots,N$ with $Y\in\St(p,n)$.

This study considers Newton's method for the JD problem on the Stiefel manifold.
The Hessian of the objective function is fundamental to deriving Newton's equation, which is the key equation in Newton's method.
We intensively examine the Hessian to efficiently solve Newton's equation.
In particular, we use the Kronecker product and the vec and veck operators to reduce the dimension of Newton's equation and to transform the equation into the standard form of $Ax=b$.

This paper is organized as follows.
We introduce the JD problem on the Stiefel manifold in Section \ref{JD_sec} as a continuation of  \cite{Theis:2009:SDR:1532023.1532070}.
In Section \ref{newton_sec}, we consider Newton's equation, which is directly obtained by substituting the gradient and Hessian formulas from Section \ref{JD_sec} into $\Hess f(Y)[\xi]=-\grad f(Y)$.
To derive the representation matrix formula of the Hessian of the objective function, we use the Kronecker product and the vec and veck operators.
This results in a smaller equation that is easier to solve.
In particular, the dimension of the resultant equation is equal to the dimension of the Stiefel manifold in question.
Thus, the equation can be solved more efficiently than the original Newton's equation.
Section \ref{ICA_sec} provides information about the three types of numerical experiments used to evaluate our method.
The first experiment is an application to ICA wherein we demonstrate that the proposed method improves the accuracy of an approximate solution to appropriately estimate the source signals.
The second experiment uses larger problems to show that sequences generated by the proposed Newton's method converge quadratically
and give better solutions than those generated by a Jacobi-like method \cite{cardoso1996jacobi}, which is an existing method for the JD problem.
Furthermore, we propose the trust-region method based on our discussion in Section \ref{newton_sec}.
The last experiment compares the proposed trust-region method with the existing method \cite{Theis:2009:SDR:1532023.1532070} and shows 
that the proposed trust-region method converges faster. Section \ref{Conclusion_sec} contains our concluding remarks.
In this study, the Stiefel manifold is endowed with the induced metric from the natural inner product in ambient Euclidean space.
In contrast to this, we derive another formula for the representation matrix of the Hessian in Appendix \ref{canonical} in which the Stiefel manifold is endowed with the canonical metric.
The resultant representation matrix under the canonical metric is symmetric, while the representation matrix with respect to the induced metric is not symmetric.

Throughout the paper, we use the following notation.
The identity matrix of size $n$ is denoted by $I_n$.
For an arbitrary matrix $W = (w_{ij})$, $W^T$ denotes the transposition of $W$ and
$\norm{W}_F := \sqrt{\sum_{i,j}w_{ij}^2}$ is the Frobenius norm of $W$.
Assume that $W$ is a square matrix;
$\sym(W) := \left(W+W^T\right)/2$ and $\skew(W) := \left(W-W^T\right)/2$ denote the symmetric and  skew-symmetric parts of $W$, respectively;
$\diag(W)$ denotes the diagonal part of $W$, that is, the $(i,j)$-th component of $\diag(W)$ is $w_{ij}\delta_{ij}$, where $\delta_{ij}$ is the Kronecker delta.
For a manifold $\mathcal{M}$, the tangent space of $\mathcal{M}$ at $x \in \mathcal{M}$ is denoted by $T_x \mathcal{M}$.
For manifolds $\mathcal{M}$ and $\mathcal{N}$ and a mapping $F: \mathcal{M} \to \mathcal{N}$, the differential of $F$ at $x \in \mathcal{M}$ is denoted by $\D F(x)$, which is a mapping from $T_x \mathcal{M}$ to $T_{F(x)}\mathcal{N}$.
The gradient and Hessian of $F$ are denoted by $\grad F$ and $\Hess F$, respectively.

\section{Joint diagonalization problem on the Stiefel manifold}\label{JD_sec}
Let $A_1,A_2,\ldots,A_N$ be $N$ real $n\times n$ symmetric matrices.
We consider the following JD problem on the Stiefel manifold $\St(p,n)$ according to \cite{Theis:2009:SDR:1532023.1532070}:

\begin{Prob}\label{JDprob}
\vspace{-\baselineskip}
\begin{align}
{\rm minimize} \,\,\,\,\,& f(Y)=-\sum_{l=1}^N\norm{\diag(Y^TA_lY)}_F^2,\label{objective}\\
{\rm subject\,\,to} \,\,\,\,\,& Y\in \St(p,n),
\end{align}
\end{Prob}
\noindent
where $\St(p,n)=\left\{Y\in\mathbb R^{n\times p}\,|\, Y^TY=I_p\right\}$ with $p\le n$.

The Hessian $\Hess f$ of $f$ is fundamental to applying Newton's method to Problem \ref{JDprob}.
To derive and analyze the Hessian and other requisites, we first review the geometry of $\St(p,n)$ as discussed in \cite{AbsMahSep2008,edelman1998geometry}.

The tangent space $T_Y\!\St(p,n)$ of $\St(p,n)$ at $Y\in\St(p,n)$ is
\begin{equation}
T_Y\!\St(p,n)=\left\{\xi\in\mathbb R^{n\times p}\,|\,\xi^TY+Y^T\xi=0\right\}.\label{tan1}
\end{equation}
In later sections, we will utilize the equivalent form \cite{edelman1998geometry}
\begin{equation}
T_Y\!\St(p,n)=\left\{YB+Y_{\perp}C\,|\,B\in\Skew(p),\ C\in\mathbb R^{(n-p)\times p}\right\},\label{tan2}
\end{equation}
rather than Eq.~\eqref{tan1}, where $Y_{\perp}$ is an arbitrary $n\times (n-p)$ matrix that satisfies $Y^TY_{\perp}=0$ and $Y_{\perp}^TY_{\perp}=I_{n-p}$, and $\Skew(p)$ denotes the set of all $p\times p$ skew-symmetric matrices.
Here, we note
\begin{equation}
K:=\dim(\St(p,n))=\frac{p(p-1)}{2}+p(n-p)=\dim(\Skew(p))+\dim(\mathbb R^{(n-p)\times p})
\end{equation}
is an important relation for rewriting Newton's equation into a system of $K$ linear equations.

Since $\St(p,n)$ is a submanifold of the matrix Euclidean space $\mathbb R^{n\times p}$, it can be endowed with the Riemannian metric
\begin{equation}
\langle \xi_1,\xi_2\rangle_Y:=\tr\left(\xi_1^T\xi_2\right),\qquad\xi_1,\xi_2\in T_Y\!\St(p,n),\label{St_met}
\end{equation}
which is induced from the natural inner product in $\mathbb R^{n\times p}$.
We view $\St(p,n)$ as a Riemannian submanifold of $\mathbb R^{n\times p}$ with the above metric.
Under this metric, the orthogonal projection $P_Y$ at $Y$ onto $T_Y\!\St(p,n)$ is expressed as
\begin{equation}
P_Y(W)=W-Y\sym\left(Y^TW\right),\qquad Y\in\St(p,n),\ W\in \mathbb R^{n\times p}.\label{projection}
\end{equation}

In optimization algorithms in the Euclidean space, a line search is performed after computing the search direction.
In Riemannian optimization, the concept of a straight line is replaced with a curve (not necessarily geodesic) on a general Riemannian manifold.
A retraction on the manifold in question is required to implement Riemannian optimization algorithms \cite{AbsMahSep2008}.
The retraction  defines an appropriate curve for searching for the next iteration point  on the manifold.
We will use the QR retraction $R$ on the Stiefel manifold $\St(p,n)$ \cite{AbsMahSep2008}, which is defined as
\begin{equation}
R_Y(\xi)=\qf(Y+\xi),\qquad Y\in\St(p,n),\ \xi\in T_Y\!\St(p,n),
\end{equation}
where $\qf(\cdot)$ denotes the Q factor of the QR decomposition of the matrix.
In other words, if a full-rank matrix $W\in\mathbb R^{n\times p}$ is uniquely decomposed into $W=QR$, where $Q\in\St(p,n)$ and $R$ is a $p\times p$ upper triangular matrix with strictly positive diagonal entries, then $\qf(W)=Q$.

To describe Newton's equation for Problem \ref{JDprob}, we need the gradient $\grad f$ and Hessian $\Hess f$ of the objective function $f$ on $\St(p,n)$.
Newton's equation is defined at each $Y\in \St(p,n)$ as $\Hess f(Y)[\xi]=-\grad f(Y)$,
where $\xi\in T_Y\!\St(p,n)$ is an unknown tangent vector.

Let $\bar f$ be the function on $\mathbb R^{n\times p}$ defined in the same way as the right-hand side of \eqref{objective}.
We note that $f$ is the restriction of $\bar f$ to $\St(p,n)$.
Expressions for the gradient and Hessian are given in \cite{Theis:2009:SDR:1532023.1532070} as
\begin{equation}
\grad f(Y)=P_Y(\grad \bar f(Y)),
\end{equation}
and
\begin{align}
\Hess f(Y)[\xi]=&P_Y\left(\D (\grad f)(Y)[\xi]\right)\\
=&P_Y(\D(\grad \bar f)(Y)[\xi]-\xi\sym\left(Y^T\grad \bar f(Y)\right)),
\end{align}
where $\grad \bar f$ is the Euclidean gradient of $\bar f$ on $\mathbb R^{n\times p}$, which is computed as
\begin{equation}
\grad \bar f(Y)=-4\sum_{l=1}^N A_lY\diag\left(Y^TA_lY\right),
\end{equation}
and where the Frech\'{e}t derivative $\D(\grad \bar f)(Y)[\xi]$ is written as
\begin{equation}
\D(\grad \bar f)(Y)[\xi]=-4\sum_{l=1}^N(A_l\xi\diag(Y^TA_lY)+2A_lY\diag(Y^TA_l\xi)).
\end{equation}
Thus, we obtain a more concrete expression for $\Hess f(Y)[\xi]$:
\begin{align}
\Hess f(Y)[\xi]=&-4\sum_{l=1}^NP_Y\left(A_l\xi\diag(Y^TA_lY)+2A_lY\diag(Y^TA_l\xi)\right.\notag\\
&\left.\qquad\qquad\qquad-\xi\sym(Y^TA_lY\diag(Y^TA_lY)) \right).\label{standardHess}
\end{align}

\section{Newton's method for the joint diagonalization problem on the Stiefel manifold}\label{newton_sec}
\subsection{The Hessian of the objective function as a linear transformation on $\Skew(p)\times \mathbb R^{(n-p)\times p}$}
Since we have already obtained the matrix expressions of $\grad f(Y)$ and $\Hess f(Y)[\xi]$ in Section~\ref{JD_sec}, Newton's equation for Problem \ref{JDprob} at $Y\in\St(p,n)$, $\Hess f(Y)[\xi]=-\grad f(Y)$, is written as
\begin{align}
&-4\sum_{l=1}^NP_Y\left(A_l\xi\diag(Y^TA_lY)+2A_lY\diag(Y^TA_l\xi)-\xi\sym(Y^TA_lY\diag(Y^TA_lY))\right)\notag\\
&=4\sum_{l=1}^NP_Y\left(A_lY\diag(Y^TA_lY)\right).\label{Newton}
\end{align}
This equation must be solved for $\xi\in T_Y\!\St(p,n)$ given $Y$.
Equation \eqref{Newton} is complicated and difficult to solve because $\xi$ is an $n\times p$ matrix with $K:=p(p-1)/2+p(n-p)\ (<np)$ independent variables.
That is, $\xi$ must satisfy $\xi^TY+Y^T\xi=0$ because $\xi$ is in $T_Y\!\St(p,n)$.

To overcome these difficulties, we wish to obtain the representation matrix of $\Hess f(Y)$ as a linear transformation on $T_Y\!\St(p,n)$ for an arbitrarily fixed $Y$
so that we can rewrite \eqref{Newton} into a standard linear equation.
To this end, we identify $T_Y\!\St(p,n)\simeq \Skew(p)\times \mathbb R^{(n-p)\times p}$ as $\mathbb R^{K}$ and view $\xi$ as a $K$-dimensional vector
using the form in Eq.~\eqref{tan2}.
We arbitrarily fix $Y_{\perp}$ to satisfy $Y^TY_{\perp}=0$ and $Y_{\perp}^TY_{\perp}=I_{n-p}$.
Such $Y_{\perp}$ can be computed by applying the Gram--Schmidt orthonormalization process to $n-p$ linearly independent column vectors of the matrix $I_p-YY^T$.
In practice, we can use MATLAB's \verb+qr+ function to obtain $Y_{\perp}$.
Using function \verb+qr+ with input $Y \in \St(p,n)$ returns an orthogonal matrix $Q \in O(n)$ such that $Y=Q\begin{pmatrix}R_1^T&0\end{pmatrix}^T$, where $R_1$ is a $p\times p$ upper triangular matrix.
If we partition $Q=\begin{pmatrix}Q_1 & Q_2\end{pmatrix}$ with $Q_1\in \St(p,n)$ and $Q_2 \in \St(n-p,n)$, then $Q_2^TY=\begin{pmatrix}0 & I_{n-p}\end{pmatrix}\begin{pmatrix}R_1^T & 0\end{pmatrix}^T=0$.
Therefore, we can choose $Q_2$ as $Y_{\perp}$.

With $Y_{\perp}$, $\xi\in T_Y\!\St(p,n)$ can be expressed as
\begin{equation}
\xi=YB+Y_{\perp}C,\qquad B\in\Skew(p),\ C\in\mathbb R^{(n-p)\times p}.\label{xi_expression}
\end{equation}
Moreover, $\Hess f(Y)[\xi]\in T_Y\!\St(p,n)$ can be written as
\begin{equation}
\Hess f(Y)[\xi]=YB_H+Y_{\perp}C_H, \qquad B_H\in\Skew(p),\ C_H\in\mathbb R^{(n-p)\times p},\label{hess_expression}
\end{equation}
and we can write $B_H$ and $C_H$ using $B$ and $C$.

\begin{Prop}
Let $Y\in\St(p,n)$ and $Y_{\perp}\in\St(n-p,n)$ satisfy  $Y^TY_{\perp}=0$.
If a tangent vector $\xi\in T_Y\!\St(p,n)$ is expressed as \eqref{xi_expression}, then the Hessian $\Hess f(Y)$ of the objective function \eqref{objective} acts on $\xi$ as $\Hess f(Y)[\xi]=YB_H+Y_{\perp}C_H$ with 
\begin{align}
B_H=&-4\sum_{l=1}^N\skew\left((Z_lB+Z_l^{\perp}C)\diag(Z_l)\right.\notag\\
&\left.\qquad\qquad\qquad+2Z_l\diag(Z_lB+Z_l^{\perp}C)-B\sym(Z_l\diag(Z_l)) \right),\label{BH}
\end{align}
and
\begin{align}
C_H=&-4\sum_{l=1}^N\left(((Z_l^{\perp})^TB+Z_l^{\perp\perp}C)\diag(Z_l)\right.\notag\\
&\left.\qquad\qquad\qquad+2(Z_l^{\perp})^T\diag(Z_lB+Z_l^{\perp}C)-C\sym(Z_l\diag(Z_l))\right),\label{CH}
\end{align}
where we have defined $Z_l=Y^TA_lY$, $Z_l^{\perp}=Y^TA_lY_{\perp}$, and $Z_l^{\perp\perp}=Y_{\perp}^TA_lY_{\perp}$.
\end{Prop} 
\begin{proof}
We first note that $Y^TP_Y(W)=\skew(Y^TW)$ and $Y_{\perp}^TP_Y(W)=Y_{\perp}^TW$ for any $W \in \mathbb R^{n\times p}$.
Multiplying Eq.~\eqref{hess_expression} by $Y^T$ from the left and using the relations $Y^TY = I_p$ and $Y^TY_{\perp} = 0$ yields
\begin{align}
&B_H=Y^T\Hess f(Y)[\xi]\notag\\
=&-4\sum_{l=1}^N\skew\left(Y^T(A_l\xi\diag(Y^TA_lY)+2A_lY\diag(Y^TA_l\xi)\right.\notag\\
&\left.\qquad\qquad\qquad-\xi\sym(Y^TA_lY\diag(Y^TA_lY))) \right).\label{BHproof}
\end{align}
Similarly, we multiply \eqref{hess_expression} by $Y_{\perp}^T$ from the left to obtain
\begin{align}
&C_H=Y_{\perp}^T\Hess f(Y)[\xi]\notag\\
=&-4\sum_{l=1}^NY_{\perp}^T\left(A_l\xi\diag(Y^TA_lY)+2A_lY\diag(Y^TA_l\xi)\right.\notag\\
&\left.\qquad\qquad\qquad-\xi\sym(Y^TA_lY\diag(Y^TA_lY)) \right).\label{CHproof}
\end{align}
Eqs.~\eqref{BH} and \eqref{CH} immediately follow from \eqref{xi_expression}, \eqref{BHproof}, and \eqref{CHproof}.
\end{proof}

\subsection{Kronecker product and the vec and veck operators}\label{Subsec_Kronecker}
The vec operator and the Kronecker product are useful for rewriting a matrix equation by  transforming the matrix into an unknown column vector \cite{neudecker1999matrix, schott2005matrix}.
The vec operator $\vec(\cdot)$ acts on a matrix $W=(w_{ij})\in\mathbb R^{m\times n}$ as
\begin{equation}
\vec(W)=\begin{pmatrix}w_{11},\ldots,w_{m1},w_{12},\ldots,w_{m2},\ldots,w_{1n},\ldots,w_{mn}\end{pmatrix}^T.
\end{equation}
That is, $\vec(W)$ is an $mn$-dimensional column vector obtained by vertically stacking the columns of $W$. 
The Kronecker product of $U\in\mathbb R^{m\times n}$ and $V\in\mathbb R^{p\times q}$ (denoted by $U\otimes V$) is an $mp\times nq$ matrix defined as
\begin{equation}
U\otimes V=\begin{pmatrix}u_{11}V&\cdots&u_{1n}V\\ \vdots & \ddots & \vdots\\ u_{m1}V &\cdots& u_{mn}V\end{pmatrix}.
\end{equation}

The following useful properties of these operators are known:
\begin{itemize}
\item
For $U\in\mathbb R^{m\times p},\ V\in\mathbb R^{p\times q}$, and $\ W\in\mathbb R^{q\times n}$,
\begin{equation}
\vec(UVW)=(W^T\otimes U)\vec(V).
\end{equation}
\item
There exists an $n^2\times n^2$ permutation matrix $T_{n}$ such that
\begin{equation}
\vec(W^T)=T_{n}\vec(W),\qquad W\in\mathbb R^{n\times n}.\label{tneq}
\end{equation}
Specifically, $T_{n}$ is given by
\begin{equation}
T_{n}=\sum_{i,j=1}^n E^{(n\times n)}_{ij}\otimes E^{(n\times n)}_{ji},\label{Tn}
\end{equation}
where $E^{(p\times q)}_{ij}$ denotes the $p\times q$ matrix that has the $(i,j)$-component equal to $1$ and all other components equal to $0$.
\end{itemize}

Furthermore, we can easily derive the following properties:
\begin{itemize}
\item
For $W\in\mathbb R^{n\times n}$,
\begin{equation}
\vec(\sym(W))=\frac{1}{2}(I_{n^2}+T_n)\vec(W),\quad \vec(\skew(W))=\frac{1}{2}(I_{n^2}-T_n)\vec(W).\label{vecsym}
\end{equation}
\item
Let $\Delta_n$ be an $n^2\times n^2$ diagonal matrix defined by $\Delta_n=\sum_{i=1}^n E^{(n\times n)}_{ii}\otimes E^{(n\times n)}_{ii}$. We have
\begin{equation}
\vec(\diag(W))=\Delta_n\vec(W),\qquad W\in\mathbb R^{n\times n}.
\end{equation}
\end{itemize}

For $C\in\mathbb R^{(n-p)\times p}$ in Eq.~\eqref{xi_expression}, $\vec(C)$ is an appropriate vector expression of $C$ because all elements of $C$ are independent variables.
On the other hand, for $B\in\Skew(p)$ in Eq.~\eqref{xi_expression}, $\vec(B)$ contains $p$ zeros stemming from the diagonal elements of $B$, which should be removed.
In addition, $\vec(B)$ contains duplicates of each independent variable because the upper triangular part (excluding the diagonal) of $B$ is the negative of the lower triangular part.
Therefore, we use the veck operator \cite{grafarend2006linear}.
The veck operator $\veck(\cdot)$ acts on $n\times n$ skew-symmetric matrix $S$ as
\begin{equation}
\veck(S)=\begin{pmatrix}s_{21},\ldots, s_{n1}, s_{32}, \ldots, s_{n2}, \ldots, s_{n,n-1}\end{pmatrix}^T.
\end{equation}
That is, $\veck(S)$ is an $n(n-1)/2$-dimensional column vector obtained by stacking the columns of the lower triangular part of $S$.
Let $D_n$ be an $n^2\times n(n-1)/2$ matrix defined by
\begin{equation}
D_n=\sum_{n\ge i>j\ge 1}\left(E^{(n^2\times n(n-1)/2)}_{n(j-1)+i,\ j(n-(j+1)/2)-n+i}-E^{(n^2\times n(n-1)/2)}_{n(i-1)+j,\ j(n-(j+1)/2)-n+i}\right).\label{Dn}
\end{equation}
Then, $D_n$ only depends on $n$ (the size of $S$) and satisfies
\begin{equation}
\vec(S)=D_n\veck(S).\label{Seq}
\end{equation}
Note that Eq.~\eqref{Seq} is valid only if $S$ is skew-symmetric.
Because each column of $D_n$ contains exactly one $1$ and one $-1$ and because each row of $D_n$ has at most one non-zero element, we have $D_n^TD_n=2I_{n(n-1)/2}$.
It follows that
\begin{equation}
\veck(S)=\frac{1}{2}D_n^T\vec(S).
\end{equation}
Furthermore, since $D_n^T\vec(W)=0$ for any $n\times n$ symmetric matrix $W$, it follows from Eq.~\eqref{vecsym} that
\begin{equation}
D_n^T(I_{n^2}+T_n)\vec(U)=2D_n^T\vec(\sym(U))=0
\end{equation}
for an arbitrary matrix $U$.
Since $U$ is arbitrary, $\vec(U)$ is also an arbitrary $n^2$-dimensional column vector. Therefore, we have $D_n^T(I_{n^2}+T_n)=0$, that is,
\begin{equation}
D_n^T=-D_n^T T_n.\label{dp_eq}
\end{equation}

\subsection{Representation matrix of the Hessian and Newton's equation}
We regard the Hessian $\Hess f(Y)$ at $Y\in\St(p,n)$ as a linear transformation $H$ on $\mathbb R^{K}$ that transforms a $K$-dimensional vector $\begin{pmatrix}\veck(B)^T & \vec(C)^T\end{pmatrix}^T$ into $\begin{pmatrix}\veck(B_H)^T & \vec(C_H)^T\end{pmatrix}^T$,
where $\xi=YB+Y_{\perp}C$ and $\Hess f(Y)[\xi]=YB_H+Y_{\perp}C_H$.
Our goal is to obtain the representation matrix of $H$.

\begin{Prop}\label{prop:rep}
Let $K:=\dim\St(p,n)=p(p-1)/2+p(n-p)$.
Let $H$ be a linear transformation on $\mathbb R^{K}$ that acts on $\begin{pmatrix}\veck(B)^T & \vec(C)^T\end{pmatrix}^T$ with $B\in \Skew(p), C\in\mathbb R^{p(n-p)}$ as
\begin{equation}
H\begin{pmatrix}\veck(B)\\ \vec(C)\end{pmatrix}=\begin{pmatrix}\veck(B_H)\\ \vec(C_H)\end{pmatrix},
\end{equation}
where $B_H$ and $C_H$ are given in Eqs.~\eqref{BH} and \eqref{CH}.
Then, the representation matrix $H_A$ of $H$ is given by 
\begin{equation}
H_A=\begin{pmatrix}H_{11}&H_{12}\\H_{21}&H_{22}\end{pmatrix},\label{HA}
\end{equation}
where
\begin{align}
H_{11}=&-2D_p^T\sum_{l=1}^N\left(\diag(Z_l)\otimes Z_l+2(I_p\otimes Z_l)\Delta_p (I_p\otimes Z_l)\right.\notag\\
&\left.\quad\quad\quad\quad\quad-\sym(Z_l\diag(Z_l))\otimes I_p\right)D_p,\label{H11}\\
H_{12}=&-2D_p^T\sum_{l=1}^N\left(\diag(Z_l)\otimes Z_l^{\perp}+2(I_p\otimes Z_l)\Delta_p(I_p\otimes Z_l^{\perp})\right),\label{H12}\\
H_{21}=&-4\sum_{l=1}^N\left((\diag(Z_l)\otimes (Z_l^{\perp})^T+2(I_p\otimes (Z_l^{\perp})^T)\Delta_p(I_p\otimes Z_l)\right)D_p,\label{H21}\\
H_{22}=&-4\sum_{l=1}^N\left(\diag(Z_l)\otimes Z_l^{\perp\perp}+2(I_p\otimes (Z_l^{\perp})^T)\Delta_p(I_p\otimes Z_l^{\perp})\right.\notag\\
&\quad\quad\quad\quad\quad-\sym(Z_l\diag(Z_l))\otimes I_{n-p}\Bigr).\label{H22}
\end{align}
\end{Prop}

\begin{proof}
From Eqs.~\eqref{BH} and \eqref{CH} together with Eq.~\eqref{dp_eq}, $\veck(B_H)$ and $\vec(C_H)$ are calculated as follows:
\begin{align}
&\veck(B_H)\notag\\
=&\frac{1}{2}D_p^T\vec\left(-4\sum_{l=1}^N\skew\left((Z_lB+Z_l^{\perp}C)\diag(Z_l)+2Z_l\diag(Z_lB+Z_l^{\perp}C)\right.\right.\notag\\
&\quad\quad\quad\quad\quad-B\sym(Z_l\diag(Z_l)) \Bigr)\Biggr)\notag\\
=&-D_p^T(I_{p^2}-T_p)\sum_{l=1}^N\left((\diag(Z_l)\otimes Z_l)\vec(B)+(\diag(Z_l)\otimes Z_l^{\perp})\vec(C)\right.\notag\\
&\left.\qquad\quad\quad+2(I_p\otimes Z_l)\vec(\diag(Z_lB+Z_l^{\perp}C))-(\sym(Z_l\diag(Z_l))\otimes I_p)\vec(B) \right)\notag\\
=&H_{11}\veck(B)+H_{12}\vec(C),
\end{align}
and
\begin{align}
&\vec(C_H)\notag\\
=&\vec\left(-4\sum_{l=1}^N\left(((Z_l^{\perp})^TB+Z_l^{\perp\perp}C)\diag(Z_l)+2(Z_l^{\perp})^T\diag(Z_lB+Z_l^{\perp}C)\right.\right.\notag\\
&\quad\quad\quad\quad\quad-C\sym(Z_l\diag(Z_l))\Bigr)\Biggr)\notag\\
=&-4\sum_{l=1}^N\left((\diag(Z_l)\otimes (Z_l^{\perp})^T)\vec(B)+(\diag(Z_l)\otimes Z_l^{\perp\perp})\vec(C)\right.\notag\\
& \left. \quad+2(I_p\otimes (Z_l^{\perp})^T)\vec(\diag(Z_lB+Z_l^{\perp}C))-(\sym(Z_l\diag(Z_l))\otimes I_{n-p})\vec(C)\right)\notag\\
=&H_{21}\veck(B)+H_{22}\vec(C).
\end{align}
This completes the proof.
\end{proof}

Thus, Newton's equation, $\Hess f(Y)[\xi]=-\grad f(Y)$, can be solved by the following method.
We first note that Newton's equation is equivalent to
\begin{align}
\begin{cases}
&Y^T\Hess f(Y)[\xi]=-Y^T\grad f(Y),\\
&Y_{\perp}^T\Hess f(Y)[\xi]=-Y_{\perp}^T\grad f(Y).\label{newtoneq}
\end{cases}
\end{align}
Applying the veck operator to the first equation of Eq.~\eqref{newtoneq} and applying the vec operator to the second equation, yields 
\begin{equation}
H_A\begin{pmatrix}\veck(B)\\ \vec(C)\end{pmatrix}=-\begin{pmatrix}\veck(Y^T\grad f(Y))\\ \vec(Y_{\perp}^T\grad f(Y))\end{pmatrix},\label{rep_new_eq}
\end{equation}
where $\xi=YB+Y_{\perp}C$ with $B\in\Skew(p)$ and $C\in\mathbb R^{(n-p)\times p}$.
If $H_A$ is invertible, we can solve Eq.~\eqref{rep_new_eq} as
\begin{equation}
\begin{pmatrix}\veck(B)\\ \vec(C)\end{pmatrix}=-H_A^{-1}\begin{pmatrix}\veck(Y^T\grad f(Y))\\ \vec(Y_{\perp}^T\grad f(Y))\end{pmatrix}.
\end{equation}
After we have obtained $\veck(B)$ and $\vec(C)$, we can easily reshape $B\in\Skew(p)$ and $C\in\mathbb R^{(n-p)\times p}$.
Therefore, we can calculate the solution $\xi=YB+Y_{\perp}C$ of Newton's equation \eqref{Newton}.

\subsection{Newton's method}\label{Subsec_Newton}
If the block matrices $H_{11}, H_{12}, H_{21},$ and $H_{22}$ of $H_A$ are related, we may reduce the computational cost of computing $H_A$.
Furthermore, if $H_A$ is symmetric, we can apply an efficient Krylov subspace method, e.g., the conjugate residual method \cite{saad2003iterative}.
The Hessian $\Hess f(Y)$ is symmetric with respect to the metric $\langle\cdot,\cdot\rangle_Y$.
However, the representation matrix $H_A$ is not always symmetric.
For $\xi=YB_1+Y_{\perp}C_1$ and $\eta=YB_2+Y_{\perp}C_2$ with $B_1, B_2\in\Skew(p)$ and $C_1, C_2\in \mathbb R^{(n-p)\times p}$, we have
\begin{equation}
\langle \xi, \eta \rangle_Y=\tr(B_1^TB_2)+\tr(C_1^TC_2)=2\veck(B_1)^T\veck(B_2)+\vec(C_1)^T\vec(C_2),\label{quadBC}
\end{equation}
so that the independent coordinates of $B_1$ and $B_2$ are counted twice.
If we endowed $\St(p,n)$ with the canonical metric \cite{edelman1998geometry}, the representation matrix would be symmetric (see Appendix \ref{canonical} for more details).

Although the representation matrix $H_A$ with the induced metric is not symmetric, it does satisfy the following proposition.
\begin{Prop}
The block matrices $H_{11}, H_{12}, H_{21},$ and $H_{22}$ defined by \eqref{H11}, \eqref{H12}, \eqref{H21}, and \eqref{H22}, respectively, satisfy
\begin{equation}
H_{11}=H_{11}^T,\quad H_{21}=2H_{12}^T,\quad H_{22}=H_{22}^T.\label{relation}
\end{equation}
\end{Prop}

The result immediately follows from Eqs.~\eqref{H11}--\eqref{H22}.
We now derive \eqref{relation} using another method to clarify how the structure of the representation matrix of the Hessian is inherited from the symmetric structure of the original Hessian.
For the function $f$ defined by \eqref{objective} under the induced metric $\langle\cdot,\cdot\rangle$ on $\St(p,n)$, we have
\begin{equation}
\langle\Hess f(Y)[\xi],\eta\rangle_Y=\langle\Hess f(Y)[\eta],\xi\rangle_Y,\qquad \xi,\eta\in T_Y\!\St(p,n),\label{hessin}
\end{equation}
because $\Hess f(Y)$ is symmetric with respect to the induced metric.
Let $\xi=YB_1+Y_{\perp}C_1$ and $\eta=YB_2+Y_{\perp}C_2$.
From Prop.~\ref{prop:rep} and Eqs.~\eqref{quadBC} and \eqref{hessin}, it follows that
\begin{equation}
\begin{pmatrix}\veck(B_1)\\ \vec(C_1)\end{pmatrix}^T\left(H_A^T\bm{J}_p-\bm{J}_pH_A\right)\begin{pmatrix}\veck(B_2)\\ \vec(C_2)\end{pmatrix}=0,
\end{equation}
where $\bm{J}_p:=\begin{pmatrix}2I_{p(p-1)/2} & 0 \\ 0 & I_{p(n-p)}\end{pmatrix}$.
Hence, we obtain
\begin{equation}
\bm{J}_pH_A=H_A^T\bm{J}_p\label{rep_sym}
\end{equation}
because $\begin{pmatrix}\veck(B_1)^T& \vec(C_1)^T\end{pmatrix}^T$ and $\begin{pmatrix}\veck(B_2)^T & \vec(C_2)^T\end{pmatrix}^T$ can be arbitrary $K$-dimensional vectors.
We can rewrite Eq.~\eqref{rep_sym} using the block matrices of $H_A$ as
\begin{equation}
\label{H_eq}
\begin{pmatrix}2H_{11}& 2H_{12}\\ H_{21} & H_{22}\end{pmatrix}=\begin{pmatrix}2H_{11}^T & H_{21}^T\\ 2H_{12}^T & H_{22}^T\end{pmatrix}.
\end{equation}
Therefore, the block matrices satisfy \eqref{relation}.
Note that this derivation does not depend on the form of function $f$.
Moreover, this result holds for any smooth function on the Stiefel manifold with the induced metric.

Using the QR retraction, we propose Algorithm \ref{Newton_alg} as Newton's method for Problem \ref{JDprob}.

\begin{algorithm}[H]
\caption{Newton's method for Problem \ref{JDprob}}
\label{Newton_alg}
\begin{algorithmic}[1]
\STATE Choose an initial point $Y^{(0)}\in \St(p,n)$.
\FOR{$k=0,1,2,\ldots$}
\STATE Compute $Y^{(k)}_{\perp}$ that satisfies $(Y^{(k)})^TY^{(k)}_{\perp}=0$ and $(Y^{(k)}_{\perp})^TY^{(k)}_{\perp}=I_{n-p}$.
\STATE Compute $Z^{(k)}_l=(Y^{(k)})^TA_lY^{(k)}$, $Z^{\perp(k)}_l=(Y^{(k)})^TA_lY^{(k)}_{\perp}$, and $Z^{\perp\perp(k)}_l=(Y^{(k)}_{\perp})^TA_lY^{(k)}_{\perp}$ for $l=1,2,\ldots,N$.
\STATE Compute $(Y^{(k)})^T\grad f(Y^{(k)})$ and $(Y^{(k)}_{\perp})^T\grad f(Y^{(k)})$ by
\begin{equation}
(Y^{(k)})^T\grad f(Y^{(k)})=-4\skew\left(\sum_{l=1}^N\left(Z^{(k)}_l\diag(Z^{(k)}_l)\right)\right),
\end{equation}
and
\begin{equation}
(Y^{(k)}_{\perp})^T\grad f(Y^{(k)})=-4\sum_{l=1}^N\left(\left(Z^{\perp(k)}_l\right)^T\diag(Z^{(k)}_l)\right).
\end{equation}
\STATE Compute the matrices $H^{(k)}_{11}$, $H^{(k)}_{12}$, and $H^{(k)}_{22}$ using \eqref{H11}, \eqref{H12}, and \eqref{H22}, respectively, with $Z_l=Z_l^{(k)}$, $Z_l^{\perp}=Z_l^{\perp(k)}$, $Z_l^{\perp\perp}=Z_l^{\perp\perp(k)}$, and compute $H_{21}^{(k)}=2(H^{(k)}_{12})^T$.
\STATE Compute $\bm b^{(k)}\in\mathbb R^{p(p-1)/2}$ and $\bm c^{(k)}\in\mathbb R^{p(n-p)}$ using
\begin{equation}
\label{bcsolve}
\begin{pmatrix}\bm b^{(k)}\\ \bm c^{(k)}\end{pmatrix}=-\begin{pmatrix}H^{(k)}_{11}&H^{(k)}_{12}\\H^{(k)}_{21}&H^{(k)}_{22}\end{pmatrix}^{-1}\begin{pmatrix}\veck((Y^{(k)})^T\grad f(Y^{(k)}))\\ \vec((Y^{(k)}_{\perp})^T\grad f(Y^{(k)}))\end{pmatrix}.
\end{equation}
\STATE Construct $B^{(k)}\in\Skew(p)$ and $C^{(k)}\in\mathbb R^{(n-p)\times p}$ that satisfy $\veck(B^{(k)})={\bm b^{(k)}}$ and $\vec(C^{(k)})=\bm c^{(k)}$.
\STATE Compute $\xi^{(k)}=Y^{(k)}B^{(k)}+Y^{(k)}_{\perp}C^{(k)}$.
\STATE Compute the next iteration $Y^{(k+1)}=\qf({Y^{(k)}}+\xi^{(k)})$.
\ENDFOR
\end{algorithmic}
\end{algorithm}
If $n=p$, that is, if we consider the case of the orthogonal group, the relationship $YY^T=I_n$ and the fact that $Y_{\perp}$ is empty simplify the algorithm.

We have thus obtained an algorithm for the JD problem with quadratic convergence based on Riemannian Newton's method.
However, because Newton's method is not guaranteed to have global convergence, we need to prepare an approximate solution to the problem or use another method such as the trust-region method.
We will discuss this in detail in Section \ref{ICA_sec}.

We conclude this section with remarks on some methods for solving Newton's equation and checking the positive definiteness of the Hessian.
First of all, $\Hess f(Y)$ is a symmetric operator with respect to the inner product $\langle \cdot, \cdot \rangle_Y$.
When $\Hess f(Y)$ is positive definite, the conjugate gradient (CG) method can be used to solve the equation if we regard the matrix--vector multiplication in the CG algorithm as the operation of $\Hess f(Y)$ on a tangent vector.
When we solve Newton's equation as a linear equation of standard form, that is, by \eqref{bcsolve},
the representation matrix $H_A$ is not necessarily symmetric.
If the dimension of the problem is large, it is difficult to solve the equation by direct inversion.
In this case, some Krylov subspace methods such as the generalized minimal residual (GMRES) method and biconjugate gradient stabilized (BiCGSTAB) method can be used.
They do not require that the coefficient matrix be symmetric.
Furthermore, we can also derive the representation matrix of the Hessian as a symmetric matrix with respect to another Riemannian metric, which is called the canonical metric in \cite{edelman1998geometry}.
A detailed derivation can be found in Appendix \ref{canonical}.
If the Hessian operator or its representation matrix is symmetric but not positive definite, we can use the conjugate residual (CR) method.
In addition, in all the cases, we can use preconditioning methods to transform the original linear operator to a better-conditioned one if the linear equation to be solved is ill-conditioned, that is, the condition number of the Hessian is large.
We refer to \cite{saad2003iterative} for further details of these Krylov subspace methods.

We now discuss the positive definiteness of the representation matrix $H_A$, which can be used to check whether a critical point obtained by Newton's method is a local minimum.
We note the following relation:
\begin{equation}
\langle\Hess f(Y)[\xi],\xi\rangle_Y=\begin{pmatrix}\veck(B_1)\\ \vec(C_1)\end{pmatrix}^T\bm{J}_p H_A\begin{pmatrix}\veck(B_1)\\ \vec(C_1)\end{pmatrix}.
\end{equation}
Therefore, if the symmetric matrix $\bm{J}_pH_A$ is positive definite, then $\Hess f(Y)$ is positive definite, and $Y$ is a local minimum.
In particular, if $n=p$, the positive definiteness of $H_A$ itself implies that $Y$ is a local minimum.
The matrix $\bm{J}_pH_A$ can be written by using block matrices as the left-hand side of \eqref{H_eq}.
By considering vectors $\begin{pmatrix} z_1^T & 0\end{pmatrix}^T$ and $\begin{pmatrix} 0 & z_2^T\end{pmatrix}^T$, we can easily derive a necessary condition for the positive definiteness of $\bm{J}_pH_A$ as
\begin{equation}
H_{11} \text{ and } H_{22} \text{ are positive definite}.
\end{equation}
We can obtain from this fact an easy way to check the necessity of the positive definiteness as
\begin{equation}
\text{All the diagonal components of $H_{11}$ and $H_{22}$ are positive.}
\end{equation}
We now consider the case where we arrive at a critical point of $f$ and assume that $\bm{J}_pH_A$ is semi-positive definite.
Then, $\bm{J}_pH_A$ is positive definite if and only if $0 \neq \det(\bm{J}_pH_A) = \det(2H_{11})\det(H_{22}-H_{21}(2H_{11})^{-1}(2H_{12}))$.
Therefore, under the assumption of semi-positive definiteness of $\bm{J}_pH_A$, the condition of positive definiteness of $\bm{J}_pH_A$ is necessary and sufficient for
\begin{equation}
\det(H_{11}) \neq 0 \text{ and } \det(H_{22}-H_{21}H_{11}^{-1}H_{12}) \neq 0.
\end{equation}

\section{Numerical experiments and application to independent component analysis}\label{ICA_sec}
In this section, we perform all of the following numerical experiments 
using MATLAB R2014b on a PC with Intel Core i7-4790 3.60 GHz CPU,
16GB of RAM memory, and Windows 8.1 Pro 64-bit operating system.
We deal with the ICA by JD with application to image separation problems in Section \ref{Sec_4_2} and general larger JD problems in Section \ref{Sec_4_3}.
We note that the ICA has wider applications in brain imaging, econometrics, image feature extraction, and so on \cite{hyvarinen2004independent}, where the proposed algorithm can be applied.
\subsection{Independent component analysis and the joint diagonalization problem}\label{Sec_4_1}
The simplest ICA model assumes the existence of $n$ independent signals $s_1(t), s_2(t),$ $\ldots, s_n(t)$.
The observations of $n$ mixtures $x_1(t), x_2(t), \ldots,x_n(t)$ are given by the mixing equation $\bm x(t)=A\bm s(t)$,
where $\bm x(t)=(x_1(t), x_2(t), \ldots,x_n(t))^T$, $\bm s(t)=(s_1(t), s_2(t),$ $\ldots,s_n(t))^T$, and $A$ is an $n\times n$ mixing matrix.
The problem is to recover the source vector $\bm s$ using only the observed data $\bm x$ under the assumption that the entries $s_1, s_2, \ldots,s_n$ of $\bm s$ are mutually independent.
The problem is formulated as the computation of an $n\times n$ matrix $B$, which is called a separating matrix, such that $\bm z(t)=B\bm x(t)$ is an appropriate estimate of the source vector $\bm s(t)$.
In other words, we wish to find $B$ such that the elements $z_1, z_2, \ldots,z_n$ of $\bm z$ are mutually independent.
See \cite{cardoso1998blind} for more details.

The ICA problem is often solved by minimizing an objective function, called a contrast function.
One choice for such a function is the JADE (joint approximate diagonalization of eigen-matrices) contrast function $\phi$, which is the sum of fourth-order cross-cumulants of the elements $z_1,z_2,\ldots,z_n$ of $\bm z$.
We can assume that $\bm x$, and therefore $\bm z$, are zero-mean random variables because we can subtract the mean $\E[\bm x]$ from $\bm x$ if needed.
The fourth-order cumulants $\mathcal{C}_{ijkl}[\bm z]$ of zero-mean random variables $z_i,z_j,z_k,z_l$ can be expressed by
\begin{equation}
\mathcal{C}_{ijkl}[\bm z]=\E[z_iz_jz_kz_l]-\E[z_iz_j]\E[z_kz_l]-\E[z_iz_k]\E[z_jz_l]-\E[z_iz_l]\E[z_jz_k].
\end{equation}
The JADE contrast function $\phi$ of $\bm z$ is then defined as
\begin{equation}
\phi(\bm z)=\sum_{i,j,k,l \atop i\neq j}(\mathcal{C}_{ijkl}[\bm z])^2.
\end{equation}
To reformulate the problem as a JD problem, we define cumulant matrices according to \cite{cardoso1999high,cardoso1993blind}.
The cumulant matrix $Q^{\bm z}(M)$ associated with a given $n\times n$ matrix $M=(m_{ij})$ is defined to have the $(i,j)$-th component
\begin{equation}
(Q^{\bm z}(M))_{ij}=\sum_{k,l=1}^n\mathcal{C}_{ijkl}[\bm z]m_{kl}.
\end{equation}
If we assume that $\bm z$ is whitened, that is, the covariance matrix of $\bm z$ is the identity matrix, then the cumulant matrix $Q^{\bm z}(M)$ can be expressed as
\begin{equation}
Q^{\bm z}(M)=\E[(\bm z^TM\bm z)\bm z \bm z^T]-\tr(M)I_n-M-M^T.
\end{equation}
Owing to the assumption of whiteness, we only have to seek a separating matrix $B$ in the orthogonal group $O(n)$.
Then, using $\bm z=B\bm x$, we can show that \cite{cardoso1999high,cardoso1993blind}
\begin{equation}
\phi(\bm z)=\sum_{k\le l}\norm{\off(Q^{\bm z}(M_{kl}))}_{F}^2=\sum_{k\le l}\norm{\off(BQ^{\bm x}(M_{kl})B^T)}_{F}^2,
\end{equation}
where $\off(\cdot)$ denotes the off-diagonal part of the matrix, and 
\begin{equation}
M_{kl}=\begin{cases}
E^{(n\times n)}_{kl}\qquad \text{if} \quad k=l\\
(E^{(n\times n)}_{kl}+E^{(n\times n)}_{lk})/\sqrt{2}\qquad \text{if} \quad k<l.
\end{cases}
\end{equation}
Therefore, if we set $N:=n(n+1)/2$ matrices $A_1,A_2,\ldots,A_N$ as $Q^{\bm x}(M_{kl}), k\le l$ and define $Y=B^T$, the optimization problem for the JADE contrast is given as follows:
\begin{Prob}\label{ICAprob}
\begin{align}
{\rm minimize} \,\,\,\,\,& g(Y)=\sum_{l=1}^N\norm{\off(Y^TA_lY)}_F^2,\label{ICAobj}\\
{\rm subject\,\,to} \,\,\,\,\,& Y\in O(n).
\end{align}
\end{Prob}
Since $Y\in O(n)$, it follows from $\norm{Y^TA_lY}_F = \norm{A_l}_F$ that
\begin{equation}
\norm{\off(Y^TA_lY)}_F^2=\norm{A_l}_{F}^2-\norm{\diag(Y^TA_lY)}_{F}^2=-\norm{\diag(Y^TA_lY)}_{F}^2+\text{const}.
\end{equation}
Thus, Problem \ref{ICAprob} is equivalent to Problem \ref{JDprob} with $p=n$.
In the next subsection, we apply Algorithm \ref{Newton_alg} to ICA.

\subsection{Application to image separation}\label{Sec_4_2}
ICA can be applied to image separation \cite{farid1999separating}.
Without loss of generality, we can assume the zero-mean property and whiteness.
Specifically, we just need to replace the observed data $\bm{x}$ with $\sqrt{\Lambda}^{-1}P^T(\bm{x}-\E[\bm{x}])$, where $P\Lambda P^T$ is the eigenvalue decomposition of the covariance matrix $\E\left[(\bm{x}-\E[\bm{x}])(\bm{x}-\E[\bm{x}])^T\right]$ of $\bm{x}$ with $P \in O(n)$ and $\Lambda$ being a diagonal matrix.
The zero-mean property $\E\left[\sqrt{\Lambda}^{-1}P^T(\bm{x}-\E[\bm{x}])\right]=0$ is obvious.
We can also verify that the covariance matrix of $\sqrt{\Lambda}^{-1}P^T(\bm{x}-\E[\bm{x}])$ is
\begin{align}
\E\left[\left(\sqrt{\Lambda}^{-1}P^T(\bm{x}-\E[\bm{x}])\right) \left(\sqrt{\Lambda}^{-1}P^T(\bm{x}-\E[\bm{x}])\right)^T\right] = & \sqrt{\Lambda}^{-1}P^T\left(P\Lambda P^T\right)P\sqrt{\Lambda}^{-1} \notag\\
= & I_n.
\end{align}

We use the $n:=12$ images obtained from \cite{ICALAB} shown in Fig.~\ref{fig:source} and expressed by $128\times 128$ matrices, $I_1,I_2,\ldots,I_{n}$.

\begin{figure}[H]
  \begin{center}
   \includegraphics[width=\columnwidth]{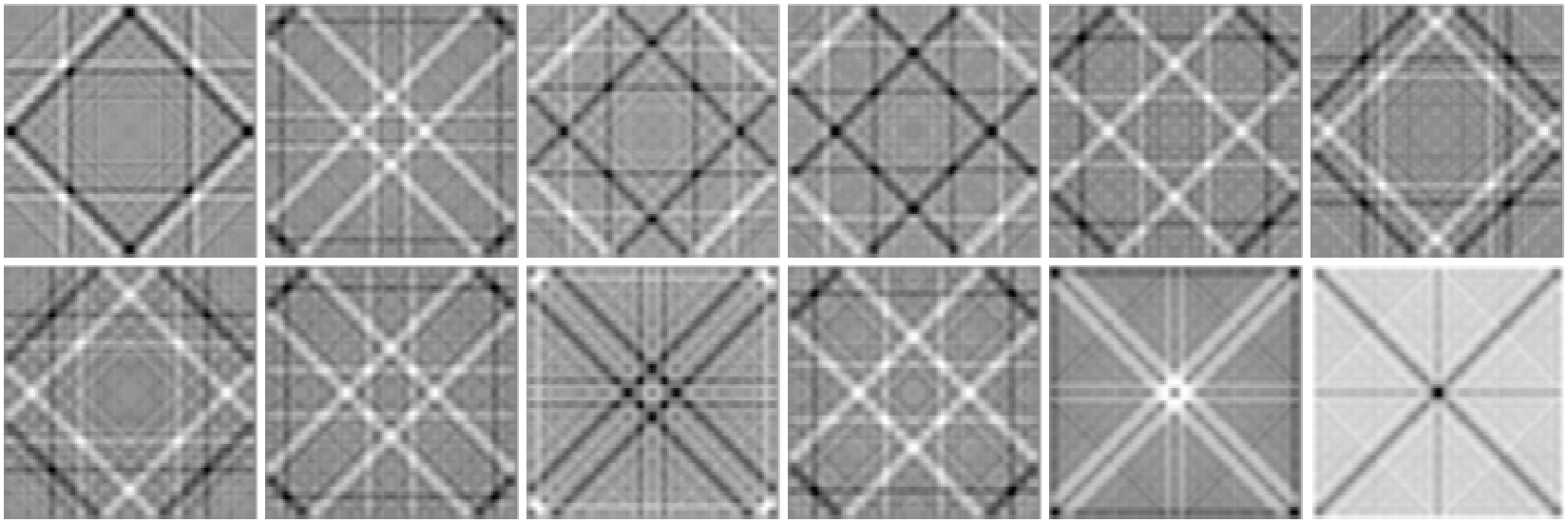}
  \end{center}
  \caption{Test images as source signals.}
  \label{fig:source}
\end{figure}

We regard the $n=12$ images as mutually independent signals according to the following discussion.
We let $s_i:=\vec(I_i), i=1,2,\ldots,n$ denote $T\ (:=128^2)$-dimensional column vectors and let $s_i(t)$ denote the $t$-th element of $s_i$.
That is, each $s_i$ has $T$ samples.
Furthermore, we define the source matrix $S:=(s_1,s_2,\ldots,s_{n})^T\in\mathbb R^{n\times T}$.
We then mix the source signals using a mixing matrix $A\in\mathbb R^{n\times n}$ to obtain $X=AS$ as observed signals.
In our experiment, $A$ is randomly chosen such that the sum of the elements of each row vector of $A$ is $1$.
The images of the observed signals $X$ are shown in Fig.~\ref{fig:mix}.
\begin{figure}[H]
  \begin{center}
   \includegraphics[width=\columnwidth]{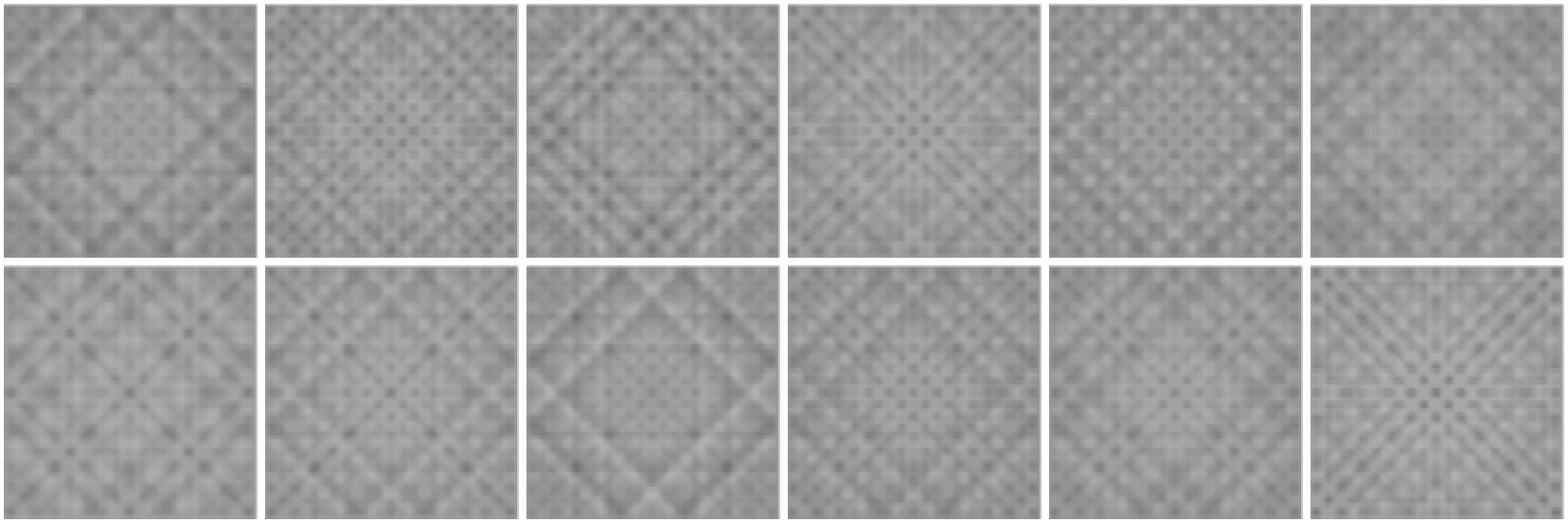}
  \end{center}
  \vspace{-0.4\baselineskip}
  \caption{Mixed images caused by the mixing matrix $A$.}
  \label{fig:mix}
\end{figure}

We wish to find a separating matrix $B$ without using any information from $S$ so that $Z:=BX$ is as mutually independent as possible.
We first compute $N:=n(n+1)/2=78$ matrices $A_1=Q^X(M_{11}), A_2=Q^X(M_{12}), \ldots, A_{N}=Q^X(M_{nn})$, as discussed in Section \ref{Sec_4_1}.
Here, we regard the operation $\E[\cdot]$ as the sample mean.
Then, we must jointly diagonalize $A_1,A_2,\ldots, A_N$, that is, solve Problem \ref{ICAprob}.
Since Newton's method only has local convergence, we need an approximate solution of the problem in advance.
In this subsection, we assume that an approximation $A_{\app}$ of the original mixing matrix $A$ is available, that is, $A_{\app}\approx A$.
We construct such $A_{\app}$ by $A_{\app}=A+0.001*\verb+randn+ (n)$.
Note that $A_{\app}^{-1}$ is not an orthogonal matrix in general.
Therefore, with $A_{\app}$ given, we compute $B_0:=\qf(A_{\app}^{-1})$.
We thus obtain an initial guess $Y_0:=B_0^T \in O(n)$ for Problem \ref{ICAprob}.

With the initial guess $Y_0$, we apply Algorithm \ref{Newton_alg} to obtain an optimal solution $Y_{\N}$ and a separating matrix $B_{\N}=Y_{\N}^T$.
After that, we compute $Z=B_{\N}X$ and estimate the separated images (Fig.~\ref{fig:estimated}) as $J_1, J_2, \ldots, J_n$ such that $\vec(J_i)$ is the $i$-th column of the $T\times n$ matrix $Z^T$ for $i=1,2,\ldots,n$.
Note that, because ICA cannot identify the correct ordering or scaling of the source signals, we have artificially ordered and scaled the estimated signals  to obtain Fig.~\ref{fig:estimated}.
\begin{figure}[htbp]
  \begin{center}
   \includegraphics[width=\columnwidth]{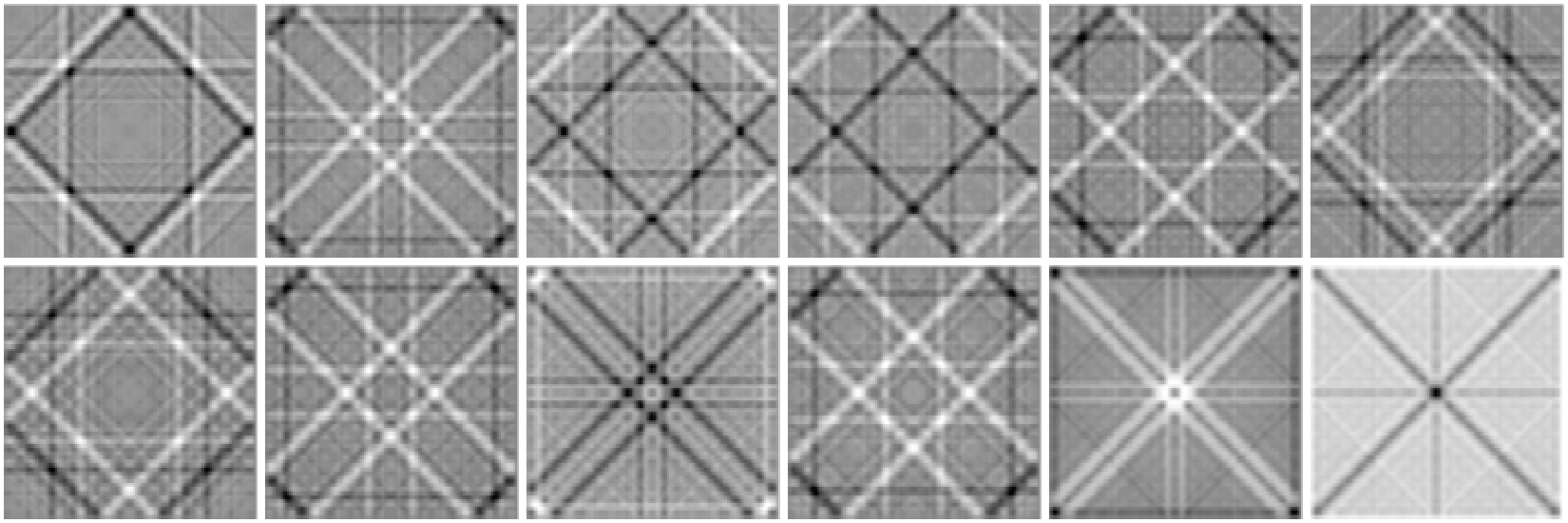}
  \end{center}
  \vspace{-0.3\baselineskip}
  \caption{Estimated images obtained by the proposed method.}
  \label{fig:estimated}
\end{figure}

The estimated images are well separated. Moreover, the proposed Newton method reduces the value of the objective function and gives a critical point by comparing our solution $Y_{\N}$ with the initial point $Y_{0}$.
The values of the objective function $g$ at $Y_{\N}$ and $Y_{0}$ are $g(Y_{\N})=64.18$ and $g(Y_{0})=67.93$.
The norm of the gradient of $g$ at $Y_{\N}$ is $\norm{\grad g(Y_{\N})}_{Y_{\N}}=7.917\times 10^{-14}$.
Furthermore, the representation matrix $H_A$ of the Hessian at $Y=Y_{\N}$ defined by \eqref{HA} and \eqref{H11} is positive definite because the smallest eigenvalue of $H_A$ is $5.7485>0$.
Note that $H_A=H_{11}$ because $p=n$.
As in Section \ref{Subsec_Newton}, positive definiteness of this matrix implies that $Y$ is a local minimum of $f$.
Thus, we can conclude that a local optimal solution has been found in our experiment though we cannot guarantee that the solution is a global minimum.

\subsection{Numerical experiments for larger problems}\label{Sec_4_3}
To more intensively investigate the performance of the proposed algorithm, we return to Problem \ref{JDprob} and consider the case $n=p=50, N=10$.
We prepare $N$ randomly chosen $n\times n$ symmetric matrices $A_1,A_2,\ldots,A_N$.
We first apply the Jacobi-like method \cite{cardoso1996jacobi} to obtain an approximate solution $Y_{\J}$.
We adopt a stopping criterion that terminates the iterative process when all the Givens rotations in a sweep have sines smaller than $\varepsilon=\verb+eps+\,(:=2^{-52})$.
We then apply the proposed Newton method by using $Y_{\J}$ as an initial point to obtain $Y_{\N}$.
The results are given as follows.
The values of the objective function $f$ are $f(Y_{\J}) = -4.12 \times 10^{3}$, $f(Y_{\N}) = -4.12 \times 10^{3}$, and $f(Y_{\J})-f(Y_{\N})=7.73\times 10^{-10}>0$.
The norms of $f$ are also compared as $\norm{\grad f(Y_{\J})}_{Y_{\J}}=5.24\times 10^{-10}$, $\norm{\grad f(Y_{\N})}_{Y_{\N}}=2.25\times 10^{-12}$, and $\norm{\grad f(Y_{\J})}_{Y_{\J}}-\norm{\grad f(Y_{\N})}_{Y_{\N}}=5.21\times 10^{-10}>0$.
Furthermore, $Y_\N$ is more orthogonal because we observe $\norm{Y_{\J}^TY_{\J}-I_p}_F-\norm{Y_{\N}^TY_{\N}-I_p}_F=1.26\times 10^{-12}>0$.
The proposed method obviously improves the accuracy of the approximate solution in this experiment.
To determine the statistical significance of the result, we run the same experiments multiple times.
As many as $1000$ experiments with sets of randomly chosen matrices $A_1,A_2,\ldots,A_N$ show that the following inequalities hold all of the time:
\begin{align}
f(Y_{\J})-f(Y_{\N})>10^{-11},\quad \norm{\grad f(Y_{\J})}_{Y_{\J}}-\norm{\grad f(Y_{\N})}_{Y_{\N}}>10^{-11},\\
\norm{Y_{\J}^TY_{\J}-I_p}_F-\norm{Y_{\N}^TY_{\N}-I_p}_F>10^{-13}.
\end{align}

We perform another experiment for $p<n$.
In this case, $n=50, p=30, N=10$, and $A_1,A_2,\ldots, A_N$ are constructed as follows.
We construct $N$ randomly chosen $n\times n$ diagonal matrices $\Lambda_1,\Lambda_2,\ldots,\Lambda_N$ and a randomly chosen $n\times n$ orthogonal matrix $P$, where the diagonal elements $\lambda_1^{(i)},\ldots,\lambda_n^{(i)}$ of each $\Lambda_i$ are positive and in descending order.
We then compute $A_1,A_2,\ldots, A_N$ as $A_i=P\Lambda_i P^T,\ i=1,2,\ldots,N$.
Note that $Y_{\opt}:=PI_{n,p}$ is an optimal solution to the problem.
We compute an approximate solution $Y_{\app}:=\qf(Y_{\opt}+Y_{\rand})$, where $Y_{\rand}$ is a randomly chosen $n\times p$ matrix that has elements less than $0.01$ (absolute values).
With $Y_{\app}$ obtained as an initial point, we apply the proposed Newton's method.
We compare the accuracy of the resultant solution $Y_{\N}$ (obtained after five iterations of the proposed method) with that of $Y_{\app}$.
The differences between the objective function $f$ and the optimal value $f(Y_{\opt})$ are
\begin{equation}
f(Y_{\app})-f(Y_{\opt})=0.136,\quad f(Y_{\N})-f(Y_{\opt})=1.42\times 10^{-14}.
\end{equation}
The norms of the gradient of the objective function $f$ satisfy
\begin{equation}
\norm{\grad f(Y_{\app})}_{Y_{\app}}=2.44,\quad \norm{\grad f(Y_{\N})}_{Y_{\N}}=2.06\times 10^{-13}.
\end{equation}

The results of these numerical experiments are presented in Table \ref{table:quadratic},
where we can observe the quadratic convergence of the sequence generated by the proposed method.
\begin{table}[htbp]
\begin{center}
\caption{Values of the objective function and norms of the gradient of the objective function obtained with five iterations of the proposed method}
\tabcolsep = 4.05pt
\begin{tabular}{c|cccccc}
\hline
$k$ & $1$ & $2$ & $3$ & $4$ &$5$ \\ \hline
$f(Y_k)-f(Y_{\opt})$ & $1.05\times 10^{-3}$ & $9.48\times 10^{-6}$ & $2.40\times 10^{-10}$ & $1.42\times 10^{-14}$ & $1.42\times 10^{-14}$\\
$\norm{\grad f(Y_k)}_{Y_k}$ & $2.04\times 10^{-2}$ & $1.09\times 10^{-3}$ & $6.03\times 10^{-6}$ & $1.37\times 10^{-10}$ & $2.06 \times 10^{-13}$\\
\hline
\end{tabular}
\label{table:quadratic}
\end{center}
\end{table}

\subsection{Application to trust-region subproblems}
When an approximate solution to the problem is not available, we may use the trust-region method instead of Newton's method because the trust-region method has global convergence.
In this subsection, we show that the results of our discussion  in Section \ref{newton_sec} can speed up solving the trust-region subproblem, and hence the performance of the trust-region method.

The trust-region method for the JD problem on the Stiefel manifold is proposed in \cite{Theis:2009:SDR:1532023.1532070}.
In the trust-region method, we use a quadratic model of the objective function $f$.
If we use the Hessian operator, we obtain a quadratic model at the $k$-th iteration:
\begin{equation}
\hat m_{Y_k}(\xi)=f(Y_k)+\langle \grad f(Y_k),\xi\rangle_{Y_k}+\frac{1}{2}\langle \Hess f(Y_k)[\xi], \xi\rangle_{Y_k}, \quad \xi \in T_{Y_k}\St(p,n).
\end{equation}
Then, the trust-region subproblem can be described as
\begin{equation}
\textrm{minimize}\quad \hat m_{Y_k}(\xi) \qquad \textrm{subject to}\quad \norm{\xi}_{Y_k}\le \Delta_k,\label{subpro}
\end{equation}
where $\Delta_k>0$ is the trust-region radius at $Y_k$.
A widely used approach for solving the trust-region subproblem is the truncated conjugate gradient (tCG) method.
See \cite{AbsMahSep2008} for the detail description of the tCG method.

Fix $k$, and let $\xi_j \in T_{Y_k}\St(p,n)$ be the $j$-th iterate of the tCG method for the subproblem expressed by \eqref{subpro}.
In the existing method, $\xi_j$ is updated as a tangent vector.
Instead, we update $B_j \in \Skew(p)$ and $C_j \in \mathbb R^{(n-p)\times p}$ by using our method presented in Section \ref{newton_sec}, where $\xi_j=Y_kB_j+(Y_k)_{\perp}C_j$.
That is, if the tCG method for \eqref{subpro} terminates with $B_J$ and $C_J$, we only have to compute $\xi_J = Y_kB_J+(Y_k)_{\perp}C_J$.
We do not have to construct $\xi_j$ for $j<J$.
Furthermore, in the existing method, $\Hess f(Y_k)[\delta_j]$ for some $\delta_j\in T_{Y_k}\St(p,n)$ must be computed at each iteration of the tCG method.
However, the proposed method needs only $(B_H)_j\in \Skew(p)$ and $(C_H)_j\in \mathbb R^{(n-p)\times p}$, where $\Hess f(Y_k)[\delta_j]=Y_k(B_H)_j+(Y_k)_{\perp}(C_H)_j$.
Given $Y_k\in\St(p,n)$, our method of solving \eqref{subpro} can be summarized as follows:
\begin{algorithm}[H]
\caption{New tCG method for the trust-region subproblem \eqref{subpro}}
\label{tCG_alg}
\begin{algorithmic}[1]
\STATE
Compute $(Y_k)_{\perp}\in \St(n-p,n)$ such that $Y_k^T(Y_k)_{\perp}=0$.
\STATE
Compute the matrices $B_g=Y_k^T\grad f(Y_k),\ C_g=(Y_k)_{\perp}^T\grad f(Y_k)$ so that $\grad f(Y_k)=Y_kB_g+(Y_k)_{\perp}C_g$, where $B_g\in\Skew(p)$ and $C_g\in\mathbb R^{(n-p)\times p}$.
Solve the following trust-region subproblem, which is equivalent to \eqref{subpro}:
\begin{equation}
\begin{aligned}
{\rm minimize} \,\,\,\,\,& \quad f(Y_k)+\left(\tr(B_g^TB)+\tr(C_g^TC)\right)+\frac{1}{2}\left(\tr(B_H^TB)+\tr(C_H^TC)\right)\\
{\rm subject\,\,to} \,\,\,\,\,& \sqrt{\tr(B^TB)+\tr(C^TC)}\le \Delta_k,
\end{aligned}
\label{subpro2}
\end{equation}
where $B_H$ and $C_H$ are defined by \eqref{BH} and \eqref{CH}.
\STATE
Let $(B_k,C_k)$ be the solution to the subproblem \eqref{subpro2} obtained in Step 2.
Compute $\xi=Y_kB_k+(Y_k)_{\perp}C_k$ as a solution to the original subproblem \eqref{subpro}.
\end{algorithmic}
\end{algorithm}

The existing method, which directly solves the subproblem \eqref{subpro}, does not contain Step 1 or Step 3.
However, Step 2 of the proposed method has a lower computational cost than directly solving \eqref{subpro}.
Therefore, if the number of iterations in the inner tCG method needed for solving the trust-region subproblems is sufficiently large,
the proposed method may have a shorter total computational time than the existing method.
These facts imply that our proposed method can reduce the computational cost.

Finally, we numerically compare the proposed method in which the trust-region subproblems are solved by Algorithm \ref{tCG_alg} with the existing method.
We fix $n=100$ and $N=5$.
For each $p\in\{10,20,\dots,90\}$, we construct $100$ sets of symmetric matrices $\{A_1,A_2,\dots, A_N\}$, solve Problem \ref{JDprob} for each set, and compute the average time needed for convergence ($\norm{\grad f(Y_k)}_{Y_k}<10^{-4}$).
\begin{table}[htbp]
\begin{center}
\caption{Computational time (seconds) between the proposed and existing trust-region methods.}
\begin{tabular}{c|ccccccccc}
\hline
$p$ & $10$& $20$ & $30$ & $40$ & $50$ & $60$ & $70$ & $80$ & $90$\\ \hline
Proposed method & $0.32$& $0.53$ & $0.92$ & $1.28$ & $1.63$ & $2.34$ & $2.83$ & $3.46$ & $4.14$\\
Existing method & $0.35$& $0.67$ & $1.04$ & $1.35$ & $1.87$ & $2.75$ & $3.68$ & $4.30$ & $5.48$\\\hline
\end{tabular}
\label{table:TR}
\end{center}
\end{table}
\begin{figure}[htbp]
  \begin{center}
   \includegraphics[width=99.55mm]{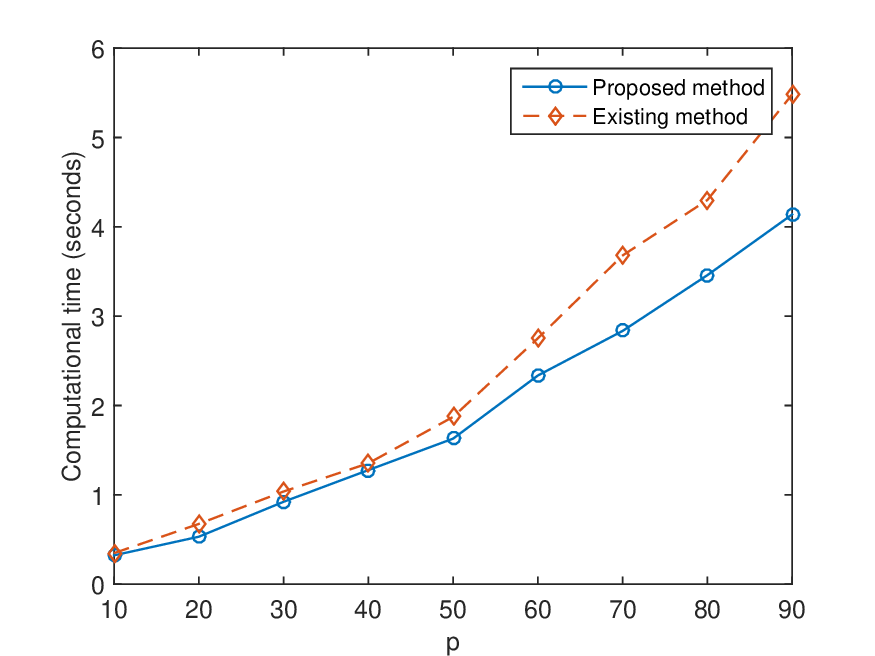}
  \end{center}
  \caption{Computational time (seconds) of the proposed and existing trust-region methods.}
  \label{fig1}
\end{figure}
Table \ref{table:TR} and Fig.~\ref{fig1} show that the proposed method is more efficient overall than the existing method.
The difference between the two methods is especially clear if $p$ is large.

\section{Concluding remarks}\label{Conclusion_sec}
We have considered the joint diagonalization problem on the Stiefel manifold $\St(p,n)$ and have developed Newton's method for the problem.
Newton's equation, $\Hess f(Y)[\xi]=-\grad f(Y)$, is difficult to solve in its original form because we must find an unknown $n\times p$ matrix $\xi$ as a tangent vector to the manifold, that is, under the condition $\xi^TY+Y^T\xi=0$.
To resolve this, we have computed the representation matrix of the Hessian of the objective function using the Kronecker product and the vec and veck operators.
The representation matrix is a $\dim (\St(p,n))\times \dim (\St(p,n))$ symmetric matrix, and we have succeeded in reducing Newton's equation into the standard form with dimension $\dim(\St(p,n))$, which is less than $np$.
Therefore, the resultant equation can be efficiently solved.
With this reduced equation, we have developed a new algorithm for the JD problem.

Furthermore, we have performed numerical experiments to verify that the present algorithm is competent for practical applications and that the algorithm has quadratic convergence.
Specifically, we have applied the proposed method to the image separating problem as an example of independent component analysis and have solved larger problems to more clearly understand the algorithm performance.
In addition, we have proposed a new trust-region method in which trust-region subproblems are solved by the truncated conjugate gradient method based on our expressions of the Hessian of the objective function.
We have observed that the proposed trust-region method is faster than the existing method.

\section*{Acknowledgements}
The author would like to thank the anonymous referees for their valuable comments that helped improve the paper significantly.

\begin{appendix}
\section{Newton's equation for Problem \ref{JDprob} with respect to the canonical metric}\label{canonical}
If the representation matrix $H_A$ is symmetric, we can apply an efficient Krylov subspace method, e.g., the conjugate residual method \cite{saad2003iterative}, to the linear equation \eqref{rep_new_eq}.
In Section \ref{newton_sec}, we have endowed the Stiefel manifold $\St(p,n)$ with the induced metric from the natural inner product in $\mathbb R^{n\times p}$.
In this section, we endow $\St(p,n)$ with another metric $g$ defined by
\begin{equation}
g_Y(\xi,\eta)=\tr\left(\xi^T\left(I_n-\frac{1}{2}YY^T\right)\eta\right),\qquad Y\in\St(p,n),\ \xi,\eta\in T_Y\!\St(p,n),
\end{equation}
which is called the canonical metric on the Stiefel manifold \cite{edelman1998geometry}.
Let $Y\in\St(p,n)$ and $Y_{\perp}\in\St(n-p,n)$ satisfy $Y^TY_{\perp}=0$.
If we let $\xi=YB_1+Y_{\perp}C_1$ and $\eta=YB_2+Y_{\perp}C_2$ with $B_1,B_2\in \Skew(p)$ and $C_1, C_2\in\mathbb R^{(n-p)\times p}$, then \begin{equation}
g_Y(\xi,\eta)=\veck\left(B_1\right)^T\veck(B_2)+\vec(C_1)^T\vec(C_2).
\end{equation}
Thus, the representation matrix of the Hessian with respect to the canonical metric should be a symmetric matrix.
We shall derive the formula for the representation matrix in a manner similar to that in Section \ref{newton_sec}.

The gradient and the Hessian of $f$ on $\St(p,n)$ depend on the metric.
For clarity, let $\grad^{\c} f$ and $\Hess^{\c} f$ respectively denote the gradient and the Hessian of $f$ with respect to the canonical metric $g$.
Let $\bar f$ be an extension of $f$ to $\mathbb R^{n\times p}$.
According to \cite{edelman1998geometry}, the gradient $\grad^{\c} f$ and the Hessian quadratic form $g_Y(\Hess^{\c} f(Y)[\xi],\eta)$ are 
\begin{align}
\grad^{\c} f(Y)=&\grad \bar f(Y)-Y\left(\grad \bar f(Y)\right)^TY\notag\\
=&-4\sum_{l=1}^N \left(A_lY\diag(Y^TA_lY)-Y\diag(Y^TA_lY)Y^TA_lY\right),
\end{align}
and
\begin{equation}
g_Y\left(\Hess^{\c} f(Y)[\xi],\eta\right)=\tr\left(G_Y(\xi)^T\eta\right),
\end{equation}
where we have defined $G_Y$ by
\begin{align}
G_Y(\xi)=&\D \left(\grad \bar f\right)(Y)[\xi]+\frac{1}{2}\Bigl(Y\xi^T\grad \bar f(Y)+\grad \bar f(Y)\xi^TY\notag\\
&\qquad\qquad\qquad\qquad\qquad-\left(I_n-YY^T\right)\xi\left(Y^T\grad \bar f(Y)+\grad \bar f(Y)^TY\right)\Bigr)\notag\\
=&-4\sum_{l=1}^N\left(A_l\xi\diag(Y^TA_lY)+2A_lY\diag(Y^TA_l\xi)\right.\notag\\
&\left.\qquad\qquad\qquad+\frac{1}{2}\left(Y\xi^TA_lY\diag(Y^TA_lY)+A_lY\diag(Y^TA_lY)\xi^TY\right)\right.\notag\\
&\left.\qquad\qquad\qquad-(I_n-YY^T)\xi\sym(Y^TA_lY\diag(Y^TA_lY))\right),
\end{align}
and $\grad \bar f$ is the standard Euclidean gradient of $\bar f$ (as in Section \ref{newton_sec}).
We can easily show that the orthogonal projection \eqref{projection} with respect to the induced metric $\langle\cdot,\cdot\rangle$ from the natural inner product is also the  orthogonal projection with respect to the canonical metric $g$.
Using this fact with the relation $(I_n+YY^T)(I_n-YY^T/2)=I_n$, we obtain
\begin{align}
g_Y\left(\Hess^{\c} f(Y)[\xi],\eta\right)=&\tr\left(G_Y(\xi)^T\eta\right)=\langle P_Y(G_Y(\xi)), \eta\rangle_Y\notag\\
=&\tr\left(P_Y(G_Y(\xi))^T(I_n+YY^T)\left(I_n-\frac{1}{2}YY^T\right)\eta\right)\notag\\
=&g_Y\left((I_n+YY^T)P_Y(G_Y(\xi)), \eta\right)
\end{align}
for arbitrary $\eta \in T_Y \!\St(p,n)$.
Because $(I_n+YY^T)P_Y(G_Y(\xi))$ is a tangent vector at $Y\in\St(p,n)$, we have
\begin{equation}
\Hess^{\c} f(Y)[\xi]=(I_n+YY^T)P_Y(G_Y(\xi)).
\end{equation}

Here, we set  $\xi=YB+Y_{\perp}C$ and $\Hess^{\c} f(Y)[\xi]=YB^{\c}_H+Y_{\perp}C^{\c}_H$, where $B, B^{\c}_H\in\Skew(p)$ and $C, C^{\c}_H\in\mathbb R^{(n-p)\times p}$.
Let $Z_l=Y^TA_lY$, $Z_l^{\perp}=Y^TA_lY_{\perp}$ and $Z_l^{\perp\perp}=Y_{\perp}^TA_lY_{\perp}$.
Then, $B^{\c}_H$ and $C^{\c}_H$ can be written as
\begin{align}
B^{\c}_H=&Y^T\Hess^{\c} f(Y)[\xi]=2\skew\left(Y^TG_Y(\xi)\right)\notag\\
=&-8\skew\Biggl(\sum_{l=1}^N\Bigl((Z_lB+Z_l^{\perp}C)\diag(Z_l)+2Z_l\diag(Z_lB+Z_l^{\perp}C)\notag\\
&+\frac{1}{2}\left((-BZ_l+C^T(Z_l^{\perp})^T)\diag(Z_l) -Z_l\diag(Z_l)B\right)\Bigr)\Biggr),
\end{align}
and
\begin{align}
C^{\c}_H=&Y_{\perp}^T\Hess^{\c} f(Y)[\xi]=Y_{\perp}^TG_Y(\xi)\notag\\
=&-4\sum_{l=1}^N\Bigl(((Z_l^{\perp})^TB+Z_l^{\perp\perp}C)\diag(Z_l)+2(Z_l^{\perp})^T\diag(Z_lB+Z_l^{\perp}C)\notag\\
&\qquad\qquad\qquad-\frac{1}{2}(Z_l^{\perp})^T\diag(Z_l)B-C\sym(Z_l\diag(Z_l))\Bigr).
\end{align}
Therefore, we obtain
\begin{equation}
\begin{pmatrix}\veck(B^{\c}_H)\\ \vec(C^{\c}_H)\end{pmatrix}=H^{\c}_A\begin{pmatrix}\veck(B)\\ \vec(C)\end{pmatrix},
\end{equation}
where the representation matrix $H^{\c}_A$ is given by $\begin{pmatrix}H^{\c}_{11}&H^{\c}_{12}\\H^{\c}_{21}&H^{\c}_{22}\end{pmatrix}$ with
\begin{align}
H^{\c}_{11}=&-4D_p^T\times\sum_{l=1}^N\left(\diag(Z_l)\otimes Z_l+2(I_p\otimes Z_l)\Delta_p (I_p\otimes Z_l)\right.\notag\\
&\qquad\qquad\qquad\qquad\left.-\sym(Z_l\diag(Z_l))\otimes I_{p}\right)D_p,
\end{align}
\begin{equation}
H^{\c}_{12}=-4D_p^T\sum_{l=1}^N\left(\diag(Z_l)\otimes Z_l^{\perp}+2(I_p\otimes Z_l)\Delta_p(I_p\otimes Z_l^{\perp})-\frac{1}{2}I_p\otimes \diag(Z_l)Z_l^{\perp}\right),
\end{equation}
\begin{align}
H^{\c}_{21}=-4\sum_{l=1}^N\biggl(\diag(Z_l)\otimes (Z_l^{\perp})^T+2(I_p\otimes (Z_l^{\perp})^T)\Delta_p(I_p\otimes Z_l)\notag\\
\left.-\frac{1}{2}I_p\otimes (Z_l^{\perp})^T\diag(Z_l)\right)D_p,
\end{align}
and
\begin{align}
H^{\c}_{22}=&-4\sum_{l=1}^N\left(\diag(Z_l)\otimes Z_l^{\perp\perp}+2(I_p\otimes (Z_l^{\perp})^T)\Delta_p(I_p\otimes Z_l^{\perp})\right.\notag\\
&\qquad\qquad\qquad-\sym(Z_l\diag(Z_l))\otimes I_{n-p}\Bigr).
\end{align}
Therefore, the solution $\xi$ to Newton's equation,
\begin{equation}
\Hess^{\c} f(Y)[\xi]=-\grad^{\c} f(Y),
\end{equation}
is $\xi=YB+Y_{\perp}C$, where $B$ and $C$ satisfy
\begin{equation}
\begin{pmatrix}\veck(B)\\ \vec(C)\end{pmatrix}=-(H^{\c}_A)^{-1}\begin{pmatrix}\veck(Y^T\grad^{\c}f(Y))\\ \vec(Y_{\perp}^T\grad^{\c}f (Y))\end{pmatrix}.
\end{equation}

Note that $H^{\c}_A$ should be symmetric so that
\begin{equation}
(H^{\c}_{11})^T=H^{\c}_{11},\ (H^{\c}_{12})^T=H^{\c}_{21},\ (H^{\c}_{22})^T=H^{\c}_{22}.\label{App_relation}
\end{equation}
We can also directly derive Eq.~\eqref{App_relation}.

We further emphasize that we can check the positive definiteness of the Hessian via the representation matrix $H^{\c}_A$ because $H^{\c}_A$ is a symmetric matrix. 

\end{appendix}
\bibliographystyle{abbrv}
\bibliography{ref_JD.bib}

\end{document}